\pdfoutput=1 

\documentclass[a4paper, twoside]{article}

\usepackage{amssymb}
\usepackage{amsmath}
\usepackage{amsthm}

\usepackage{mathrsfs}

\usepackage{url}

\usepackage{microtype}

\usepackage{hyperref}

\swapnumbers
\theoremstyle{plain}
\newtheorem{theorem}{Theorem}[section]
\newtheorem{lemma}[theorem]{Lemma}
\newtheorem{corollary}[theorem]{Corollary}

\theoremstyle{remark}
\newtheorem{remark}[theorem]{Remark}
\newtheorem{example}[theorem]{Example}

\theoremstyle{definition}
\newtheorem{definition}[theorem]{Definition}
\newtheorem{miniremark}[theorem]{}

\DeclareMathOperator{\without}{\sim}

\newcommand{\restrict}{\mathop{\llcorner}}

\DeclareMathOperator{\trace}{trace}

\DeclareMathOperator{\Bdry}{Bdry}

\DeclareMathOperator{\with}{:}      
\DeclareMathOperator{\spt}{spt}     
\DeclareMathOperator{\im}{im}       
\DeclareMathOperator{\Lip}{Lip}     
\DeclareMathOperator{\dmn}{dmn}     
\DeclareMathOperator{\Hom}{Hom}     
\DeclareMathOperator*{\aplim}{\mathrm{ap}\, \lim}   

\newcommand{\ud}{\,\mathrm{d}}

\newcommand{\tint}[2]{{\textstyle\int_{#1}^{#2}}}

\title{An isoperimetric inequality for diffused surfaces}
\author{Ulrich Menne \and Christian Scharrer}
\date{\today}

\begin{document}
\maketitle 

\begin{abstract}
	For general varifolds in Euclidean space, we prove an isoperimetric
	inequality, adapt the basic theory of generalised weakly
	differentiable functions, and obtain several Sobolev type
	inequalities.  We thereby intend to facilitate the use of varifold
	theory in the study of diffused surfaces.
\end{abstract}

\paragraph{MSC-classes 2010.}  53A07 (Primary); 46E35, 49Q15 (Secondary).


\section{Introduction}

\paragraph{General aim.} The isoperimetric inequality is well established in
the context of sharp surfaces (e.g., integral currents, sets, or integral
varifolds) in Euclidean space, but little appears to be known for diffused
surfaces (i.e., for surfaces that are not concentrated on a set of the their
own dimension).  General varifolds form a very flexible model for the latter
case; in fact, for equations of Allen-Cahn type, their utility was established
by Ilmanen, Padilla, and Tonegawa (see~\cite{MR1237490} and \cite{MR1611144})
and, for discrete and computational geometry, their unifying use has been
recently suggested by Buet, Leonardi, and Masnou
(see~\cite{arXiv:1609.03625v1}).  The present paper shall contribute to this
proposed development by adapting several core tools to the possibly
non-rectifiable case.  To outline these results, \emph{suppose $m$ and $n$ are
positive integers, $m \leq n$, $V$ is an $m$~dimensional varifold in~$\mathbf
R^n$, and, \emph{to avoid case distinctions,} also $m>1$}; see
Section~\ref{sec:notation} for the notation.

\paragraph{Isoperimetric inequality, see
Section~\ref{section:isoperimetricInequality}.} The best result up to now 
(see the second author~\cite[6.11]{scharrer:MSc}) did apply to
general varifolds, but controlled only their rectifiable parts: \emph{If $\| V
\| ( \mathbf R^n ) < \infty$, then
\begin{equation*}
	\| V \| \, \{ x \with \boldsymbol \Theta^m ( \| V \|, x ) \geq d \}
	\leq \Gamma d^{-1/m} \| V \| ( \mathbf R^n )^{1/m} \| \delta V \| (
	\mathbf R^n ) \quad \text{for $0 < d < \infty$},
\end{equation*}
where $\Gamma$ is a positive, finite number determined by $m$}.  Following the
first author (see \cite[2.2]{MR2537022}), it unified the approach of Allard
in~\cite[7.1]{MR0307015} and Michael and Simon in~\cite[2.1]{MR0344978}.
Clearly, if $0 < d < \infty$, and $\boldsymbol \Theta^m ( \| V \|, x) \geq d$
for $\| V \|$~almost all~$x$, the result implies
\begin{equation*}
	\| V \| ( \{ x \with \boldsymbol \Theta^m ( \| V \|, x ) \geq d
	\})^{1-1/m} \leq \Gamma d^{-1/m} \| \delta V \| ( \mathbf R^n ).
\end{equation*}
We notice that $\| \delta V \|$ encodes both, the total mass of the
variational boundary and the integral of the modulus of the generalised mean
curvature of the varifold, see Allard~\cite[4.3]{MR0307015};  in particular, a
more classical form results for varifolds with vanishing mean curvature (i.e.,
generalised minimal surfaces) and, by Allard~\cite[4.8\,(4)]{MR0307015}, the
isoperimetric inequality for integral currents with non-optimal constant is a
special case.  In~\ref{theorem:improvedIsopIneq} and~\ref{definition:gamma},
we establish that, \emph{if $\| V \| ( \mathbf R^n ) < \infty$, then
\begin{align*}
	& \| V \| ( A (d ) )^{1-1/m} \leq \boldsymbol \gamma (m) d^{-1/m} \|
	\delta V \| ( \mathbf R^n ) \quad \text{for $0 < d < \infty$}; \\
	& \qquad \text{where $A(d) = \{ x \with \text{$\| V \| \, \mathbf
	B(x,r) \geq d \boldsymbol \alpha (m) r^m$ for some $0 < r < \infty$}
	\}$}.
\end{align*}}
By homogeneity considerations, one may not replace $(A(d),d^{-1/m})$ by
$(\mathbf R^n, 1 )$.  The sets~$A(d)$, for suitable $d$, naturally describe
the region, where the behaviour of the diffused surface resembles the
behaviour of an $m$~dimensional sharp surface.

\paragraph{Generalised weakly differentiable functions, see
Section~\ref{sec:tv}.}  We extend the basic theory of generalised
weakly differentiable functions (see the first
author~\cite[\S\S\,8--9]{MR3528825} and \cite[4.1,\,2]{snulmenn.sobolev}) from
rectifiable varifolds to general varifolds.  This theory includes the study of
closedness properties (under convergence, composition, addition, and
multiplication) and a coarea formula in functional analytic form.  The main
differences lie in the possible non-existence of decompositions
(see~\ref{example:decomposition}) and the ineffectiveness of $( \| V\|,
m)$~approximate differentials (see~\ref{remark:lipschitzian}).  This
development allows us to state the Sobolev inequalities in their natural
framework, but goes beyond that purpose.

\paragraph{Sobolev inequalities, see Section~\ref{sec:sobolev}.}  
In view of~\ref{remark:generalise_sect_8}, \ref{remark:generalise_sect_9}, and
\cite[8.16, 9.2]{MR3528825}, a version of our Sobolev inequality
in~\ref{thm:sob_average} may be stated as follows, employing (see
\ref{definition:genWeakDiffFct}) the space of $Y$~valued generalised weakly
differentiable functions $\mathbf T(V,Y)$ and the derivative $V\,\mathbf Df$
associated to functions $f$ in that space: \emph{If $\| \delta V \|$ is a
Radon measure, $Y$ is a finite dimensional normed vector space, $f \in \mathbf
T( V, Y )$, $\| V \| \, \{ x \with |f(x)| >  0 \} < \infty$, $0 < r < \infty$,
and $g : \mathbf R^n \to \mathbf R$ satisfies
\begin{equation*}
	g(a) = \sup \big \{ y \with \| V \| ( \mathbf B (a,r) \cap \{ x \with
	|f(x)| \leq y \} ) \leq 2^{-1}\| V \| \, \mathbf B(a,r) \big \}
\end{equation*}
for $a \in \mathbf R^n$, then, for $0 < d < \infty$, there holds
\begin{equation*}
	{\textstyle \big ( \int_{B(d)}g^{m/(m-1)} \, \mathrm d \|V \|
	\big)^{1-1/m} \leq \Gamma d^{-1/m} \big ( \int |f| \, \mathrm d \|
	\delta V \| + \int \| V \, \mathbf Df \| \, \mathrm d \| V \| \big )},
\end{equation*}
where $B(d) = \{ x \with \| V \| \, \mathbf B (x,r) \geq d \boldsymbol \alpha
(m) r^m \}$, and $\Gamma = 2 \boldsymbol \beta (n) \boldsymbol \gamma (m)$.}
In this theorem, the number $r$ acts as a scale on which both the lower
density ratio bound and the averaging process by medians occur; in fact, the
width of a diffused surface could be a natural choice for such a scale.  More
generally, in~\ref{thm:sob_average}, we replace $r$ by a $\| V \|$~measurable
function.  Simple examples show that one may not replace $(g,B(d))$ by
$(f,A(d))$, see \ref{remark:optimality_sob}.  Finally, we note that the
special case $0 \leq f \in \mathscr D ( \mathbf R^n , \mathbf R )$ of the
preceding theorem could be derived replacing the use of \ref{remark:wy} and
the coarea formula for generalised weakly differentiable functions
(see~\ref{remark:generalise_sect_8} and \cite[8.5,\,30]{MR3528825}) by
Allard's more basic result~\cite[4.10]{MR0307015}.

\paragraph{Acknowledgements.}  We would like to thank Dr~Blanche Buet,
Professor Guido De~Philippis, and Professor Yoshihiro Tonegawa for
conversations on the subject of this paper. The paper was written while both
authors worked at the Max Planck Institute for Gravitational Physics (Albert
Einstein Institute) and the University of Potsdam.

\section{Notation} \label{sec:notation}

Generally, the notation of \cite[\S\,1]{MR3528825} is employed;  the only
exception is the usage of $\boldsymbol \gamma (m)$, see~\ref{definition:gamma}
and \ref{remark:gamma}.  In particular, our notation is largely consistent
with that of Fe\-de\-rer~\cite[pp.~669--671]{MR41:1976} and
Allard~\cite{MR0307015}.  While we do not duplicate each definition from
\cite[\S\,1]{MR3528825}, for the convenience of the reader, we recall some
less commonly used symbols and conventions below.

The difference of sets $A$ and $B$ is denoted by $A \without B$.  Whenever
$f$~is a linear map and $v$ belongs to its domain, the expression $\langle v,f
\rangle$ is synonymously used with~$f(v)$.  The inner product of $v$ and $w$,
by contrast, is denoted by $v \bullet w$.  The symbol~$P_\natural$ denotes the
symmetric linear homomorphism of $\mathbf R^n$ whose image is~$P$ and whose
restriction to~$P$ is the identity map of~$P$, whenever $P$ is a linear
subspace of~$\mathbf R^n$.  If $X$~is a locally compact Hausdorff space, then
$\mathscr K (X)$ denotes the vector space of continuous real valued functions
on $X$ with compact support.  Whenever $\phi$ measures $X$,  $Y$ is a
separable Banach space, $f$ is a $\phi$~measurable $Y$~valued function, and $1
\leq p \leq \infty$, the value of the Lebesgue seminorm~$\phi_{(p)}$ at $f$
satisfies
\begin{align*}
	\phi_{(p)}(f) & = {\textstyle \big ( \int|f|^p \, \mathrm d \phi
	\big )^{1/p} \quad \text{if $p < \infty$}}, \\
	\phi_{(p)}(f) & = \inf \big \{ s \with \text{$s \geq 0$, $\phi \, \{ x
	\with |f(x)|>s \} = 0$} \big \} \quad \text{if $p = \infty$}.
\end{align*}
Whenever $U$ is an open subset of a finite dimensional normed space, and
$Y$~is a separable Banach space, $\mathscr E (U,Y)$ denotes the space of all
functions from~$U$ into~$Y$, that are continuously differentiable of every
positive integer order, $\mathscr D(U,Y)$ denotes the subspace of those
functions in $\mathscr E(U,Y)$ with compact support, and $\mathscr D' ( U,
Y)$~is the space of distributions in~$U$ of type~$Y$.  Whenever $T \in
\mathscr D' ( U, Y)$, the symbol~$\| T \|$ denotes the largest Borel regular
measure over~$U$ such that
\begin{equation*}
	\| T \| ( A ) = \sup \{ T ( \theta ) \with \text{$\theta \in \mathscr
	D ( U,Y )$ with $\spt \theta \subset A$ and $| \theta (x) | \leq 1$
	for $x \in U$} \}
\end{equation*}
whenever $A$ is an open subset of~$U$.  In case such $T$ is representable by
integration (equivalently, if $\| T \|$ is a Radon measure), $T ( \theta )$
denotes the value of the unique $\| T \|_{(1)}$ continuous extension of $T$ to
$\mathbf L_1 ( \| T \|, Y )$ at $\theta \in \mathbf L_1 ( \| T \|, Y )$, and
$T \restrict A$ denotes the restriction of $T$ to $A$, whenever $A$ is $\| T
\|$ measurable (i.e., we have $(T \restrict A)( \theta ) = T(
\theta_A )$ whenever $\theta \in \mathscr{D}( U, Y )$, where $\theta_A( x ) =
\theta( x )$ for $x \in A$ and $\theta_A( x ) = 0$ for $x \in U \without A$).
Finally, if $V$ is an $m$~dimensional varifold in an open subset $U$ of
$\mathbf R^n$, $\| \delta V \|$ is a Radon measure, and $E$ is an $\| V \| +
\| \delta V \|$~measurable set, then the distributional boundary~$V \,
\partial E$ satisfies
\begin{equation*}
	V \, \partial E = ( \delta V ) \restrict E - \delta ( V \restrict E
	\times \mathbf G (n,m) ) \in \mathscr{D}' ( U, \mathbf R^n ).
\end{equation*}

\section{General isoperimetric inequality}
\label{section:isoperimetricInequality}

In this section, we prove a general isoperimetric inequality
in~\ref{theorem:improvedIsopIneq}.  It involves a~maximal-type function
corresponding to the density defined in~\ref{miniremark:defMaxFunction}.
Additionally, its proof relies on a simple iteration lemma
(see~\ref{lemma:iteration}) and a~variant of the ``calculus lemma'' used by
Simon (see~\ref{lemma:improvedIsopIneq} and~\ref{remark:calculus_lemma}).
Finally, in~\ref{theorem:poincareTypeIneq}
and~\ref{corollary:iso_ineq_with_size}, we state a~version of the
isoperimetric inequality in case the varifold is contained in a ball and a
version involving the size of the varifold.

\begin{miniremark} [Maximal-type function] \label{miniremark:defMaxFunction}
	Suppose $m$ and $n$ are positive integers, $m \leq n$, $V \in
	\mathbf{V}_m (\mathbf{R}^n)$, and the function $M : \mathbf{R}^n \to
	\overline{\mathbf{R}}$ satisfies
	\begin{gather*}
		M(x) = \sup \left \{ \frac{\|V\| \, \mathbf{B}(a,s)}{
		\boldsymbol{\alpha}(m)s^m} \with \text{$a \in \mathbf{R}^n$,
		$0 < s < \infty$, and $x \in \mathbf{B}(a,s)$} \right \}
	\end{gather*}
	for $x \in \mathbf{R}^n$.  Clearly, if $a \in \mathbf{R}^n$, $0 < s <
	\infty$, $0 < d < \infty$, and $\| V \| \, \mathbf{B} (a,s) \geq d
	\boldsymbol \alpha (m) s^m$, then $\mathbf{B} (a,s) \subset \{ x \with
	M(x) \geq d \}$.
\end{miniremark}

\begin{lemma} [Iteration lemma] \label{lemma:iteration}
	Suppose $0 \leq \kappa < \infty$, $0 < \lambda < 1$, $0 < \mu < 1$,
	the function $a : \{ d \with 0 < d < \infty \} \to \mathbf{R}$ is
	nonnegative, $\limsup_{d \to 0+} a(d) < \infty$, and
	\begin{equation*}
		a(d) \leq \kappa d^{-\mu} a(\lambda d)^\mu \quad \text{for
		$0 < d < \infty$}.
	\end{equation*}
	
	Then, $a(d)^{1-\mu} \leq \kappa d^{-\mu} (1/\lambda)^{\mu^2/(1-\mu)}$
	for $0 < d < \infty$.
\end{lemma}

\begin{proof}
	Assume $\kappa > 0$.  Then, induction yields that $\log a(d)$ does not
	exceed
	\begin{equation*}
		 {\textstyle \big ( \log ( \kappa ) + \mu \log (1/d) \big )
		 \left ( \sum_{i=0}^{j-1} \mu^i \right ) + \log (1/ \lambda )
		 \left ( \sum_{i=1}^{j-1} i \mu^{i+1} \right ) + \mu^j \log a
		 ( \lambda^j d )}
	\end{equation*}
	whenever $0 < d < \infty$ and $j$ is a positive integer; here
	$\sum_{i=1}^0 i \mu^{i+1} = 0$.
\end{proof}

\begin{lemma} [Calculus lemma] \label{lemma:improvedIsopIneq}
	Suppose $0 < s < \infty$, $f : \{ t \with s \leq t < \infty \} \to
	\mathbf{R}$ is a nonnegative, nondecreasing function, $0 < m <
	\infty$, $s^{-m} f(s) \geq 3/4$,
	\begin{equation*}
		r = \sup \{ t \with \text{$s \leq t < \infty$ and $t^{-m} f(t)
		\geq 1/3$} \} < \infty,
	\end{equation*}
	$g : \{ t \with s \leq t \leq r \} \to \mathbf{R}$ is a nonnegative,
	$\mathscr{L}^1 \restrict \{ t \with s \leq t \leq r \}$ measurable
	function, and $t^{-m} f(t) \leq r^{-m} f(r) + \textstyle \int_t^r
	u^{-m} g(u) \,\mathrm d \mathscr{L}^1 \, u$ whenever $s \leq t \leq
	r$.

	Then, there exists $t$ satisfying
	\begin{equation*}
		s \leq t \leq r \quad \text{and} \quad f (5t) \leq 5^m r g(t).
	\end{equation*}
\end{lemma}

\begin{proof}
	Abbreviating $\kappa = \sup \{ t^{-m} f(t) \with s \leq t \leq r \}$,
	we note that
	\begin{gather*}
		s \leq r, \quad 3/4 \leq \kappa < \infty, \quad \sup \{ t^{-m}
		f(t) \with r \leq t < \infty \} \leq 1/3, \\
		\tint sr t^{-m} f(t) \ud \mathscr{L}^1 \, t \leq
		\kappa r, \quad \tint r{5r} t^{-m} f(t) \ud \mathscr{L}^1 \, t
		\leq {\textstyle \frac {4r}3}.
	\end{gather*}
	Therefore, if the conclusion were false, we could estimate
	\begin{gather*}
		\tint sr t^{-m} g (t) \ud \mathscr{L}^1 \, t < r^{-1} \tint
		sr (5t)^{-m} f(5t) \ud \mathscr{L}^1 \, t = {\textstyle
		\frac{1}{5r}} \tint{5s}{5r} t^{-m} f(t) \ud \mathscr{L}^1 \, t
		\leq {\textstyle \frac \kappa 5 + \frac {4}{15}}, \\
		\kappa \leq r^{-m} f(r) + \tint sr t^{-m} g(t) \ud
		\mathscr{L}^1 \, t < {\textstyle \frac 35 + \frac \kappa 5},
	\end{gather*}
	whence, as $3/4 \leq \kappa < \infty$, it would follow $3/5 \leq
	4\kappa/5 < 3/5$, a contradiction.
\end{proof}

\begin{remark} \label{remark:calculus_lemma}
	The previous lemma and its proof are adapted from
	\cite[18.7]{MR756417}.
\end{remark}

\begin{theorem} [General isoperimetric inequality]
	\label{theorem:improvedIsopIneq}

	Suppose $m$, $n$, $V$, and $M$ are as in
	\ref{miniremark:defMaxFunction}, and $\| V \| ( \mathbf{R}^n ) <
	\infty$.  Then,
	\begin{equation*}
		\| V \| ( \{ x \with M(x) \geq d \} )^{1-1/m} \leq \Gamma
		d^{-1/m} \| \delta V \| ( \mathbf{R}^n ) \quad \text{for $0 <
		d < \infty$},
	\end{equation*}
	where $\Gamma = 2^{-1}$ if $m=1$, $\Gamma = 5^m 3^{1/(m-1)}
	\boldsymbol \alpha (m)^{-1/m}$ if $m > 1$, and $0^0=0$.
\end{theorem}

\begin{proof}
	Assume $\| \delta V\| ( \mathbf{R}^n ) < \infty$.  In view of
	\cite[4.8\,(1)]{MR3528825}, we may assume that $m>1$.  We abbreviate
	$\kappa = 5^m 3^{1/m} \boldsymbol \alpha (m)^{-1/m} \| \delta V \| (
	\mathbf{R}^n )$.  By \ref{lemma:iteration} applied with $\lambda$,
	$\mu$, and $a(d)$ replaced by $1/3$, $1/m$, and $\| V \| \, \{ x \with
	M(x) \geq d \}$, respectively, it is sufficient to prove
	\begin{equation*}
		\| V \| \, \{ x \with M(x) \geq d \} \leq \kappa d^{-1/m} \| V
		\| ( \{ x \with M(x) \geq d/3 \} )^{1/m} \quad \text{for $0 <
		d < \infty$}.
	\end{equation*}
	For this purpose we define
	\begin{equation*}
		r = \sup \{ s \with \text{$a \in \mathbf{R}^n$, $0 < s <
		\infty$, and $\| V \| \, \mathbf{B} (a,s) \geq (d/3)
		\boldsymbol \alpha (m) s^m$} \}
	\end{equation*}
	and note that $r \leq 3^{1/m} d^{-1/m} \boldsymbol \alpha (m)^{-1/m}
	\| V \| ( \{ x \with M(x) \geq d/3 \} )^{1/m} < \infty$ by
	\ref{lemma:iteration}.  Moreover, whenever $x \in \mathbf{R}^n$ and
	$M(x) \geq d$, there exist $a \in \mathbf{R}^n$ and $0 < t \leq r$
	satisfying
	\begin{equation*}
		x \in \mathbf{B} (a,t), \quad \| V \| \, \mathbf{B} (a,5t)
		\leq 5^m r \| \delta V \| \, \mathbf{B} (a,t);
	\end{equation*}
	in fact, taking $a \in \mathbf{R}^n$ and $0 < s < \infty$ with $x \in
	\mathbf{B} (a,s)$ and $\| V \| \, \mathbf{B} (a,s) \geq (3d/4)
	\boldsymbol \alpha (m) s^m$, in view of \cite[4.5,\,6]{MR3528825}, one
	may apply \ref{lemma:improvedIsopIneq} with $f(t)$ and $g(t)$ replaced
	by $d^{-1} \boldsymbol \alpha (m)^{-1} \| V \| \, \mathbf{B} (a,t)$ and
	$d^{-1} \boldsymbol \alpha (m)^{-1} \| \delta V \| \, \mathbf{B}
	(a,t)$, respectively.  Finally, Vitali's covering theorem (see
	\cite[2.8.5,\,8]{MR41:1976}) yields the conclusion.
\end{proof}

\begin{remark}
	Apart of a possibly unnecessarily large number $\Gamma$, the preceding
	isoperimetric inequality comprises, firstly, that of Allard in
	\cite[7.1]{MR0307015}, secondly, that of the first author in
	\cite[2.2]{MR2537022}, and, thirdly, those of the second author in
	\cite[6.5,\,11]{scharrer:MSc}.  The last two items as well as the
	present inequality employ the strategy introduced by Michael and Simon
	in \cite[2.1]{MR0344978} (see also Simon \cite[18.6]{MR756417}) in the
	context of Sobolev inequalities.
\end{remark}

\begin{definition} [Best isoperimetric constant] \label{definition:gamma}
	Whenever $m$ is a positive integer, we denote by $\boldsymbol \gamma
	(m)$ the smallest nonnegative real number with the following property:
	if $n$, $V$, and $M$ are related to $m$ as in
	\ref{miniremark:defMaxFunction}, and $\| V \| ( \mathbf{R}^n ) <
	\infty$, then
	\begin{equation*}
		\| V \| ( \{ x \with M(x) \geq d \} )^{1-1/m} \leq \boldsymbol
		\gamma (m) d^{-1/m} \| \delta V \| ( \mathbf{R}^n ) \quad
		\text{for $0 < d < \infty$};
	\end{equation*}
	here $0^0=0$.
\end{definition}

\begin{remark}
	Considering a unit disc, we notice $\boldsymbol \alpha
	(m)^{-1/m}/m \leq \boldsymbol \gamma(m)$; in particular,
	$\boldsymbol{\gamma}(1) = 2^{-1}$ by \ref{theorem:improvedIsopIneq}.
	Also, if $m>1$, then $\boldsymbol \gamma (m) \leq 5^m 3^{1/(m-1)}
	\boldsymbol \alpha (m)^{-1/m}$ by \ref{theorem:improvedIsopIneq}, but
	the precise value of $\boldsymbol \gamma(m)$ is unknown.
\end{remark}

\begin{remark} \label{remark:gamma}
	Notice that $\boldsymbol \gamma (m)$ is greater or equal to the number
	bearing that name in \cite[\S\,1]{MR3528825}; if $m>1$, it is unknown
	whether these numbers agree.
\end{remark}

\begin{corollary} [General isoperimetric inequality in a ball]
	\label{theorem:poincareTypeIneq}
	Suppose $m$ and $n$ are positive integers, $m \leq n$, $V \in \mathbf
	V_m ( \mathbf R^n )$, $\| \delta V \|$ is a Radon measure, $a \in
	\mathbf R^n$, $0 < r < \infty$, and $\spt \| V \| \subset \mathbf B
	(a,r)$.
	
	Then,
	\begin{equation*}
		\boldsymbol{\alpha}(m)^{-1/m} r^{-1} \|V\| (\mathbf{R}^n) \leq
		\boldsymbol{\gamma}(m) \|\delta V\| (\mathbf{R}^n).
	\end{equation*}
\end{corollary}

\begin{proof}
	Letting $d = \boldsymbol \alpha (m)^{-1} r^{-m} \| V \| ( \mathbf R^n
	)$ and assuming $d > 0$, it is sufficient to apply
	\ref{definition:gamma}, since $M(x) \geq d$ for $x \in \mathbf
	B(a,r)$, see \ref{miniremark:defMaxFunction}.
\end{proof}

\begin{miniremark} [Embeddings of weak Lebesgue spaces]
\label{miniremark:weak_lp_spaces}

	If $\phi$ measures $X$, $f$ is a $\phi$~measurable $\{ t \with 0 \leq
	t \leq \infty \}$~valued function, $\phi \, \{ x \with f(x) > 0 \} <
	\infty$, $1 \leq q < p < \infty$, and
	\begin{equation*}
		\kappa = \sup \big \{ d \, \phi ( \{ x \with f (x) \geq d \}
		)^{1/p} \with 0 < d < \infty \big \} < \infty,
	\end{equation*}
	then we have $\phi_{(q)} (f) \leq (1-q/p)^{-1/q} \phi ( \{ x \with
	f(x)>0 \})^{1/q-1/p} \kappa$.
\end{miniremark}

\begin{corollary} [General isoperimetric inequality with size]
\label{corollary:iso_ineq_with_size}
	Suppose $m$ and $n$ are positive integers, $2 \leq m \leq n$, $V \in
	\mathbf V_m ( \mathbf R^n )$, and $( \| V \| + \| \delta V \| ) (
	\mathbf R^n) < \infty$.

	Then,
	\begin{equation*}
		d \, \mathscr H^m ( \{ x \with \boldsymbol \Theta^m ( \| V \|,
		x) \geq d \} )^{1-1/m} \leq \boldsymbol \gamma (m) \| \delta V
		\| ( \mathbf R^n ) \quad \text{for $0 < d < \infty$};
	\end{equation*}
	in particular, if $V \in \mathbf {RV}_m ( \mathbf R^n )$ and $\mathscr
	H^m \, \{ x \with \boldsymbol \Theta^m ( \| V \|, x )
	> 0 \} < \infty$, then
	\begin{equation*}
		\| V \| ( \mathbf R^n ) \leq m \boldsymbol \gamma (m) \mathscr
		H^m ( \{ x \with \boldsymbol \Theta^m ( \| V \|, x ) > 0 \}
		)^{1/m} \| \delta V \| ( \mathbf R^n ).
	\end{equation*}
\end{corollary}

\begin{proof}
	The principal conclusion is a consequence of~\ref{definition:gamma},
	as \cite[2.10.19\,(3)]{MR41:1976} yields
	\begin{equation*}
		\mathscr H^m \, \{ x \with \boldsymbol \Theta^m ( \| V \|, x )
		\geq d \} \leq d^{-1} \| V \| \, \{ x \with M(x) \geq d \}.
	\end{equation*}
	The postscript follows from \cite[3.5\,(1b)]{MR0307015} and
	\ref{miniremark:weak_lp_spaces} with $p = \frac{m}{m-1}$ and $q=1$.
\end{proof}

\section{Generalised weakly differentiable functions} \label{sec:tv}

This section extends the basic theory of generalised weakly differentiable
functions on rectifiable varifolds (see~\cite[\S\S\,8--9]{MR3528825}
and~\cite[4.1,\,2]{snulmenn.sobolev}) to general varifolds.  In fact, most of
it extends almost verbatim (see~\ref{remark:tv_old_and_new},
\ref{remark:composition}, \ref{remark:generalise_sect_8},
\ref{remark:generalise_sect_9}, and \ref{remark:TGV_multiplication}) once
suitable approximation procedures for Lipschitzian functions
(see~\ref{lemma:lip_approx} and~\ref{lemma:local_der}) are available to
replace the usages of the $(\| V\|, m)$~approximate differential
in~\cite[\S\S\,8--9]{MR3528825}.  It is unknown
(see~\ref{remark:open_question}), whether the role of that approximate
differential could be taken, for Lipschitzian functions, by the notion of
differentiability introduced by Alberti and Marchese in~\cite{MR3494485}.

A part that does not extend is the existence of decompositions
(see~\ref{example:decomposition}), hence the same holds for the
characterisation of functions with vanishing derivative
(see~\ref{remark:rect_necc}).  Nevertheless, a generalised weakly
differentiable function may, under the natural summability hypothesis, be
defined using a partition of the varifold induced by sets with vanishing
distributional boundary (see~\ref{thm:def_by_partition}).

\begin{lemma} [Disintegration for varifolds] \label{lemma:disintegration}
	Suppose $m$ and $n$ are positive integers, $m \leq n$, $U$ is an open
	subset of $\mathbf R^n$, and $V \in \mathbf V_m ( U )$.  Then,
	(see~\cite[3.3]{MR0307015})
	\begin{equation*}
		{\textstyle \int k \, \mathrm d V = \iint k(x,P) \, \mathrm d
		V^{(x)} \, P \, \mathrm d \| V \| \, x}
	\end{equation*}
	whenever $k$ is an $\overline{\mathbf R}$~valued $V$ integrable
	function.
\end{lemma}

\begin{proof}
	The case $k \in \mathscr K ( U \times \mathbf G(n,m) )$ is treated by
	Allard in~\cite[3.3]{MR0307015}.
	Noting~\cite[2.5.13,\,14]{MR41:1976}, successive approximation by the
	method of~\cite[2.5.3]{MR41:1976} yields the case that $k$ is a
	characteristic function of a $V$~measurable set.  Finally, we employ
	\cite[2.3.3,\,4.8,\,4.4\,(6)]{MR41:1976} to deduce the general case.
\end{proof}

\begin{definition} [Generalised~$V$~weakly differentiable functions]
	\label{definition:genWeakDiffFct}
	Suppose $m$ and $n$ are positive integers, $m \leq n$, $U$ is an open
	subset of $\mathbf{R}^n$, $V\in\mathbf{V}_m(U)$, $\|\delta V\|$ is a
	Radon measure, and $Y$ is a finite dimensional normed vector space.

	Then, a $Y$ valued $\|V\|+\|\delta V\|$ measurable function $f$ with
	$\dmn f\subset U$ is called \emph{generalised $V$ weakly
	differentiable} if and only if for some $\|V\|$ measurable
	$\Hom(\mathbf{R}^n,Y)$ valued function $F$, the following two
	conditions hold:
	\begin{enumerate}
		\item If $K$ is a compact subset of $U$ and $0\leq s<\infty$,
		then
		\begin{equation*}
			{\textstyle\int_{K\cap\{x \with |f(x)|\leq
			s\}}\|F\|\,\mathrm d\|V\|<\infty}.
		\end{equation*}
		\item If $\theta\in\mathscr{D}(U,\mathbf{R}^n)$,
		$\gamma\in\mathscr{E}(Y,\mathbf{R})$ and $\spt \mathrm
		D\gamma$ is compact, then
		\begin{equation*}
			(\delta V)((\gamma\circ f)\theta)
			=
			{\textstyle\int\gamma(f(x))P_\natural\bullet\mathrm
			D\theta(x)+
			\langle\theta(x),\mathrm D\gamma(f(x))\circ
			F(x)\rangle\,\mathrm dV\, (x,P)}.
		\end{equation*} 
	\end{enumerate}
	The function $F$ is $\|V\|$ almost unique.  Therefore, one may define
	the \emph{generalised $V$~weak derivative of~$f$} to be the function
	$V\,\mathbf{D}f$ characterised (see
	\cite[2.8.9,\,14,\,9.13]{MR41:1976}) by $a\in\dmn V\,\mathbf{D}f$ if
	and only if
	\begin{align*}
		& (\|V\|,C)\aplim_{x\to a}F(x)=\sigma\quad \text{for some
		}\sigma\in\Hom(\mathbf{R}^n,Y), \\
		& \qquad \text{where $C = \{ (a,\mathbf{B}(a,r)) \with
		\text{$a \in \mathbf R^n$, $0 < r < \infty$, and
		$\mathbf{B}(a,r) \subset U$} \}$},
	\end{align*}
	and, in this case, $V\,\mathbf{D}f(a)=\sigma$.  The set of all
	$Y$~valued generalised $V$~weakly differentiable functions will be
	denoted by $\mathbf{T}(V,Y)$.  Finally, $\mathbf{T}(V) = \mathbf{T}
	(V,\mathbf{R})$.
\end{definition}

\begin{remark}
	This definition is in accordance with \cite[8.3]{MR3528825}, where it
	is introduced under the additional hypothesis that $V$ is rectifiable.
\end{remark}

\begin{remark} \label{remark:tv_old_and_new}
	The closedness results \cite[4.1,\,2]{snulmenn.sobolev} hold (with the
	same proof) when the condition $V \in \mathbf{RV}_m ( U)$ in their
	statements is replaced by $V \in \mathbf{V}_m ( U )$.
\end{remark}

\begin{example} \label{example:c1}
	If $f: U \to Y$ is of class~$1$, then $f \in \mathbf T(V,Y)$ and
	(see~\cite[3.3]{MR0307015})
	\begin{equation*}
		{\textstyle V \, \mathbf Df(x) = \mathrm Df(x) \circ \int
		P_\natural \, \mathrm d V^{(x)} \, P \quad \text{for $\| V
		\|$~almost all~$x$};}
	\end{equation*}
	hence, $\| V \, \mathbf Df(x) \| \leq \big \| \mathrm Df(x) |
	T(x) \big \|$, where $T(x) = \im \int P_\natural \, \mathrm d V^{(x)}
	\, P$, for such~$x$.
\end{example}

\begin{lemma} [Lipschitzian functions~I]
\label{lemma:lip_approx}
	Suppose $m$, $n$, $U$, $V$, and $Y$ are as
	in~\ref{definition:genWeakDiffFct}.  Then, the following two
	statements hold.
	\begin{enumerate}
		\item \label{item:lip_approx:lip} If $f : U \to Y$ is a
		locally Lipschitzian function, then $f \in \mathbf{T} (V,Y)$
		and
		\begin{equation*}
			\| V \, \mathbf Df (x) \| \leq \lim_{r \to 0+} \Lip (
			f | \mathbf{B} (x,r) ) \quad \text{for $\| V \|$
			almost all $x$}.
		\end{equation*}
		\item \label{item:lip_approx:convergence} If $f_i : U \to Y$
		is a sequence of locally Lipschitzian functions
		converging to $f : U \to Y$ locally uniformly as $i \to
		\infty$, and
		\begin{equation*}
			\kappa = \sup \{ \| V \|_{(\infty)} ( V \, \mathbf
			Df_i ) \with i = 1, 2, 3, \ldots \} < \infty,
		\end{equation*}
		then $f \in \mathbf{T} (V,Y)$, $\| V \|_{(\infty)} ( V \,
		\mathbf Df ) \leq \kappa$, and
		\begin{equation*}
			\lim_{i \to \infty} {\textstyle \int \langle
			V\,\mathbf Df_i, G \rangle \ud \| V \| = \int \langle
			V\,\mathbf Df, G \rangle \ud \| V \|}
		\end{equation*}
		whenever $G \in \mathbf{L}_1 \big ( \| V \|, \Hom (
		\mathbf{R}^n,Y )^\ast \big )$.
	\end{enumerate}
\end{lemma}

\begin{proof}
	Let \eqref{item:lip_approx:convergence}$'$ denote the statement
	resulting from \eqref{item:lip_approx:convergence} by adding the
	hypothesis, that the sequence~$f_i$ belongs to $\mathbf T(V,Y)$.  To
	prove \eqref{item:lip_approx:convergence}$'$, by
	\cite[2.5.7\,(ii)]{MR41:1976}, \cite[2.1]{snulmenn.sobolev}, and
	\cite[V.4.2,\,5.1]{MR90g:47001a}, we may assume that, for some $F \in
	\mathbf{L}_\infty ( \| V \|, \Hom ( \mathbf{R}^n, Y ) )$ with $\| V
	\|_{(\infty)} (F) \leq \kappa $, we have
	\begin{equation*}
		\lim_{i \to \infty} {\textstyle \int \langle V\,\mathbf Df_i,
		G \rangle \, \mathrm d \| V \| = \int \langle F, G \rangle \,
		\mathrm d \| V \|} \quad \text{for $G \in \mathbf{L}_1 \big (
		\| V \|, \Hom ( \mathbf{R}^n, Y )^\ast \big )$}.
	\end{equation*}
	Then, \ref{remark:tv_old_and_new} and~\cite[4.1]{snulmenn.sobolev}
	yield \eqref{item:lip_approx:convergence}$'$ which, as we observe,
	implies~\eqref{item:lip_approx:lip} by means of convolution and
	\ref{example:c1}; in particular,
	\eqref{item:lip_approx:convergence}$'$ and
	\eqref{item:lip_approx:convergence} are equivalent.
\end{proof}

\begin{remark} \label{remark:lipschitzian}
	\eqref{item:lip_approx:lip} partly generalises \cite[8.7]{MR3528825}.
	In the remaining part thereof (i.e., in the
	characterisation of $V \, \mathbf Df$ in terms of the $(\| V \|,
	m)$~approximate differential), the hypothesis $V \in \mathbf{RV}_m ( U
	)$ may not be weakened to $V \in \mathbf V_m ( U )$; in fact,
	\ref{example:c1} and \cite[4.8\,(2)]{MR0307015} readily yield
	examples.
\end{remark}

\begin{remark} \label{remark:open_question}
	Here, we formulate two open questions which relate the preceding
	remark to the differentiability theory of Lipschitzian functions by
	Alberti and Marchese~(see \cite[1.1]{MR3494485}).  For this purpose,
	suppose $m$ and $n$ are positive integers, $m < n$, $V \in \mathbf V_m
	( \mathbf R^n )$, and $\| \delta V \|$ is a Radon measure.
	\begin{enumerate}
		\item Does it follow that, for $\| V \|$~almost all~$x$, the
		image of $q(x)=\int P_\natural \, \mathrm d V^{(x)} \, P \in
		\Hom ( \mathbf R^n, \mathbf R^n )$ is contained in the
		``decomposability bundle of~$\| V \|$'' at~$x$ introduced by
		Alberti and Marchese in \cite[2.6]{MR3494485}?
		\item If so, does it follow that, for Lipschitzian functions
		$f : \mathbf R^n \to Y$,
		\begin{equation*}
			V\, \mathbf{D} f(x) = \mathrm D (f|Q(x)) (x) \circ q
			(x) \quad \text{for $\| V \|$~almost all~$x$},
		\end{equation*}
		where $Q(x) = \{ x + \langle v,q(x) \rangle \with v \in
		\mathbf R^n \}$?
	\end{enumerate}
\end{remark}

\begin{remark}
	\ref{lemma:lip_approx}\,\eqref{item:lip_approx:convergence} is
	analogous to \cite[4.5\,(3)]{MR2898736}.
\end{remark}

\begin{remark} \label{remark:composition}
	Referring to \ref{example:c1} and
	\ref{lemma:lip_approx}\,\eqref{item:lip_approx:convergence} in place
	of \cite[4.5\,(3)]{MR2898736}, and noting \ref{lemma:disintegration},
	the result in \cite[8.6]{MR3528825} takes the following form: \emph{If
	$f \in \mathbf{T} (V,Y)$, $\theta : U \to \mathbf{R}^n$ is
	Lipschitzian with compact support, $\gamma : Y \to \mathbf{R}$ is of
	class~$1$, and either $\spt \mathrm D\gamma$ is compact or $f$ is
	locally bounded, then
	\begin{equation*}
		( \delta V) ( ( \gamma \circ f ) \theta ) {\textstyle = \int
		\gamma (f(x)) \trace ( V \, \mathbf{D} \theta (x) ) + \langle
		\theta (x), \mathrm D \gamma (f(x)) \circ V \, \mathbf{D} f(x)
		\rangle \, \mathrm d \| V \| \, x.}
	\end{equation*}
	Consequently, if $f \in \mathbf{T} (V,Y)$ is locally bounded, $Z$ is
	finite dimensional normed vector space, and $g : Y \to Z$ is of
	class~$1$, then $g \circ f \in \mathbf{T} (V,Z)$ with
	\begin{equation*}
		V \, \mathbf D(g \circ f) (x) = \mathrm Dg(f(x)) \circ V \,
		\mathbf Df(x) \quad \text{for $\| V \|$ almost all $x$}.
	\end{equation*}}
	We notice that, by~\cite[8.7]{MR3528825}, this is a generalisation
	of~\cite[8.6]{MR3528825}.
\end{remark}

\begin{remark} \label{remark:generalise_sect_8}
	The results
	\cite[8.4,\,5,\,12,\,15,\,16,\,18,\,20,\,29,\,30,\,33]{MR3528825}
	remain valid when the references to ``Definition~8.3'' in
	\cite{MR3528825} in their statements and proofs are replaced by
	references to the present, more general definition
	in~\ref{definition:genWeakDiffFct}; in fact, it is sufficient to
	additionally replace the references to ``Remark~8.6'' in their proofs
	in~\cite{MR3528825} by references to \ref{remark:composition} in the
	present paper and use (instead of ``Example~8.7'' in~\cite{MR3528825})
	an approximation based on convolution,~\ref{example:c1},
	and~\ref{lemma:lip_approx}, to justify the second ingredient to the
	equality on page~1029, line~25 in \cite{MR3528825}: namely, the
	equation
	\begin{equation*}
		\langle u, V\,\mathbf D ( ( \nu \circ g ) \theta ) (x) \rangle
		= \langle u, \mathrm D \nu (g(x)) \circ V \, \mathbf Dg(x)
		\rangle \theta (x) + \nu (g(x)) \langle u, V\, \mathbf D\theta
		(x) \rangle
	\end{equation*}
	whenever $u \in \mathbf R^n$, for $\|V \|$~almost all~$x$.
\end{remark}

\begin{example} [Nonexistence of decompositions] \label{example:decomposition}
	Suppose $m$ and $n$ are positive integers, $m < n$, and $T \in \mathbf
	G(n,m)$.  Then the following three statements hold.
	\begin{enumerate}
		\item \label{item:decomposition:mu} If $\mu$ is a Radon
		measure over~$\mathbf R^n$, $V = \mu \times \boldsymbol
		\delta_T \in \mathbf V_m ( \mathbf R^n )$, $\| \delta V \|$ is
		a Radon measure, $f : \mathbf R^n \to \mathbf R$ is of
		class~$1$ with $\mathrm Df(x)|T =0$ for $x \in \mathbf R^n$,
		and $E(y) = \{ x \with f(x)>y \}$ for $y \in \mathbf R$, then
		$V \, \partial E(y) = 0$ for $y \in \mathbf R$.
		\item \label{item:decomposition:E} If $E$ is an $\mathscr L^n$
		measurable set, $V = ( \mathscr L^n \restrict E ) \times
		\boldsymbol \delta_T \in \mathbf V_m ( \mathbf R^n )$, $\|
		\delta V \|$ is a Radon measure, and $V$ is indecomposable,
		then $V = 0$.
		\item \label{item:decomposition:no} If $V = \mathscr L^n
		\times \boldsymbol \delta_T \in \mathbf V_m ( \mathbf R^n )$,
		then $\delta V = 0$ and there does not exist a decomposition
		of~$V$.
	\end{enumerate}
	To prove \eqref{item:decomposition:mu}, we notice $f \in \mathbf T
	(V)$ and $V \, \mathbf Df(x) = 0$ for $\| V \|$~almost all~$x$ by
	\ref{example:c1}, whence we deduce the assertion by means of
	\ref{remark:generalise_sect_8} and \cite[8.29]{MR3528825}.  Moreover,
	\eqref{item:decomposition:mu} yields \eqref{item:decomposition:E} by
	taking $f$ to be a nonzero member of $\Hom ( \mathbf R^n , \mathbf R
	)$ with $T \subset \ker f$.  Finally,
	Allard~\cite[4.8\,(2)]{MR0307015} and \eqref{item:decomposition:E}
	imply \eqref{item:decomposition:no}.
\end{example}

\begin{remark} \label{remark:rect_necc}
	\ref{example:decomposition}\,\eqref{item:decomposition:no} shows that
	the rectifiability hypotheses in \cite[6.12, 8.34]{MR3528825} may not
	be omitted.
\end{remark}

\begin{theorem} [Weakly differentiable functions by partitions] \label{thm:def_by_partition}
	Suppose $m$, $n$, $U$, $V$, and $Y$ are as in
	\ref{definition:genWeakDiffFct}, $\Xi$ is a countable subset of
	$\mathbf V_m ( U )$, $\xi$~maps $\Xi$ into the class of all Borel
	subsets of~$U$ such that distinct members of $\Xi$ are mapped onto
	disjoint sets, $( \| V \| + \| \delta V \| ) \big ( U \without \bigcup
	\im \xi \big ) = 0$,
	\begin{gather*}
		\text{$W = V \restrict \xi ( W ) \times \mathbf G(n,m)$ and $V
		\, \partial \xi ( W ) = 0$} \quad \text{for $W \in \Xi$},
	\end{gather*}
	$f_W \in \mathbf T (W,Y)$ for $W \in \Xi$, and
	\begin{equation*}
		{\textstyle f = \bigcup \{ f_W | \xi ( W ) \with W \in \Xi \},
		\quad F = \bigcup \{ (W \, \mathbf D f_W ) | \xi ( W ) \with W
		\in \Xi \}}.
	\end{equation*}

	Then, the following three statements hold:
	\begin{enumerate}
		\item The function $f$ is $\| V \| + \| \delta V \|$
		measurable.
		\item The function $F$ is $\| V \|$ measurable.
		\item If $\int_{K \cap \{ x \with |f(x)| \leq s \}} \| F \| \,
		\mathrm d \| V \| < \infty$ whenever $K$ is a compact subset
		of~$U$ and $0 \leq s < \infty$, then $f \in \mathbf T(V,Y)$ and
		\begin{equation*}
			V \, \mathbf Df(x) = F (x) \quad \text{for $\| V
			\|$~almost all~$x$}.
		\end{equation*}
	\end{enumerate}
\end{theorem}

\begin{proof}
	The proof of \cite[8.24]{MR3528825} applies unchanged.
\end{proof}

\begin{remark}
	In view of \ref{remark:rect_necc}, it is important that the preceding
	generalisation of \cite[8.24]{MR3528825} does not assume the members
	of $\Xi$ to be indecomposable.
\end{remark}

\begin{definition} [Zero boundary values] \label{definition:TGV}
	Suppose $m$ and $n$ are positive integers, $m \leq n$, $U$ is an open
	subset of~$\mathbf R^n$, $V \in \mathbf V_m ( U )$, $\| \delta V\|$ is
	a Radon measure, and $G$ is a relatively open subset of~$\Bdry U$.
	Then, $\mathbf T_G (V)$ is defined to be the set of all nonnegative
	functions $f \in \mathbf T (V)$ such that, with $B = ( \Bdry U )
	\without G$ and $E(y) = \{ x \with f(x) > y \}$ for $0 < y < \infty$,
	the following conditions hold for $\mathscr{L}^1$~almost all $0 < y
	<\infty$:
	\begin{gather*}
		( \| V \| + \| \delta V \| ) (E(y) \cap K ) + \| V \, \partial
		E (y) \| ( U \cap K ) < \infty, \\
		{\textstyle\int_{E(y) \times \mathbf G (n,m)} P_\natural
		\bullet \mathrm D\theta (x) \, \mathrm d V \, (x,P) = ( (
		\delta V ) \restrict E(y) ) ( \theta | U ) - V \, \partial
		E(y) ( \theta | U )}
	\end{gather*}
	whenever $K$ is a compact subset of $\mathbf R^n \without B$ and
	$\theta \in \mathscr{D} ( \mathbf{R}^n \without B, \mathbf R^n )$.
\end{definition}

\begin{remark} \label{remark:wy}
	Defining $W_y \in \mathbf V_m ( \mathbf R^n \without B )$ by
	\begin{equation*}
		W_y ( k ) = {\textstyle \int_{E(y) \times \mathbf G (n,m)} k \,
		\mathrm d V} \quad \text{for $k \in \mathscr K ( ( \mathbf R^n
		\without B ) \times \mathbf G (n,m) )$}
	\end{equation*}
	for $0 < y < \infty$, we see that, whenever $y$ satisfies the
	conditions of \ref{definition:TGV}, we have
	\begin{equation*}
		\| \delta W_y \| ( A ) \leq \| \delta V \| ( E(y) \cap A ) +
		\| V \, \partial E(y) \| ( U \cap A ) \quad \text{for $A
		\subset \mathbf R^n \without B$};
	\end{equation*}
	in particular, $\| \delta W_y \|$ is a Radon measure for such $y$.
\end{remark}

\begin{remark} \label{remark:generalise_sect_9}
	The definition in \ref{definition:TGV} is in accordance with
	\cite[9.1]{MR3528825}, where it is stated under the additional
	hypothesis that $V$ is rectifiable.  Moreover, the results of
	\cite[9.2,\,4,\,5,\,9,\,12,\,13,\,14]{MR3528825} remain valid if the
	references to ``Definition~9.1'' in their statements and proofs
	in~\cite{MR3528825} are replaced by references to the present, more
	general definition in \ref{definition:TGV}; in fact, taking
	\ref{remark:generalise_sect_8} into account, the proofs remain
	otherwise unchanged.
\end{remark}

\begin{lemma} [Lipschitzian functions~II] \label{lemma:local_der}
	Suppose $m$, $n$, $U$, $V$, and $G$ are as in \ref{definition:TGV}, $c
	: U \cup G \to \mathbf R$, $\Lip (c|K) < \infty$ whenever $K$ is a
	compact subset of $\mathbf R^n \without B$, $g=c|U$, $D$ is a $\| V
	\|$~measurable set, $W \in \mathbf V_m ( \mathbf R^n \without B)$,
	$W(k) = \int_{D \times \mathbf G(n,m)} k \, \mathrm d V$ for $k \in
	\mathscr K ( ( \mathbf R^n \without B ) \times \mathbf G (n,m) )$, and
	$\| \delta W \|$~is a Radon measure.

	Then, there holds $c \in \mathbf T (W)$, $g \in \mathbf T (V)$, and
	\begin{equation*}
		W \, \mathbf Dc(x) = V \, \mathbf Dg (x) \quad \text{for $
		\| V \|$~almost all $x \in D$}.
	\end{equation*}
\end{lemma}

\begin{proof}
	Assuming $\kappa = \Lip c < \infty$, we extend $c$ to $\zeta : \mathbf
	R^n \to \mathbf R$ such that $\Lip \zeta = \kappa$ by means of
	\cite[2.10.43]{MR41:1976}.  Then, using convolution, we construct a
	sequence $\zeta_i \in \mathscr E( \mathbf R^n, \mathbf R )$ with $\Lip
	\zeta_i \leq \kappa$ for every positive integer~$i$ and
	\begin{equation*}
		\text{$\zeta_i (x) \to \zeta(x)$, uniformly for $x \in \mathbf
		R^n$, as $i \to \infty$}.
	\end{equation*}
	Since $W^{(x)} = V^{(x)}$ for $\| V \|$~almost all~$x \in D$ by
	\cite[2.8.9,\,18,\,9.11]{MR41:1976}, we note
	\begin{equation*}
		W \, \mathbf D(\zeta_i|\mathbf R^n \without B ) (x) = V \,
		\mathbf D ( \zeta_i|U ) (x) \quad \text{for $\| V \|$~almost
		all~$x \in D$}
	\end{equation*}
	for every positive integer~$i$ by \ref{example:c1}.  Therefore,
	passing to the limit $i \to \infty$ with the help of
	\ref{lemma:lip_approx}, we deduce the conclusion.
\end{proof}

\begin{remark} \label{remark:TGV_multiplication}
	The result of \cite[9.16]{MR3528825} remains valid if the references
	to ``Definition~9.1'' in its statement and its proof are replaced by
	references to the present, more general definition in
	\ref{definition:TGV}; in fact, taking \ref{remark:generalise_sect_8}
	and \ref{remark:generalise_sect_9} into account, it is sufficient to
	additionally replace the occurrences of ``$\mathbf{RV}_m$'' on page
	1044, lines 15 and 29 in \cite{MR3528825} by ``$\mathbf V_m$'' and the
	words ``Example~8.7 in conjunction with \cite[2.10.19\,(4),
	2.10.43]{MR41:1976}'' on page 1044, lines 26--27 in \cite{MR3528825}
	by a reference to \ref{lemma:local_der} in the present paper.
\end{remark}

\section{Sobolev inequalities} \label{sec:sobolev}

In this section, we present Sobolev inequalities for generalised weakly
differentiable functions with zero boundary values, that are entailed by the
general isoperimetric inequalities in~\ref{theorem:improvedIsopIneq}
and~\ref{theorem:poincareTypeIneq}.  As the formal analogue
to~\ref{theorem:improvedIsopIneq} does not hold
(see~\ref{example:no_formal_analogue}), two alternative formulations are
offered.  The first version (see~\ref{thm:sob_average}) involves an averaging
process based on medians and a scale (possibly depending on the point).  The
second version (see~\ref{thm:sob_rect}) implies control only on the
rectifiable part.  For both statements, we isolate a classical technique due
to Federer in~\ref{lemma:federers_fct} and~\ref{remark:federers_fct}.  The
analogue for \ref{theorem:poincareTypeIneq}, in contrast, is immediate
(see~\ref{thm:poincareTypeIneq}).  Finally, the negative results of this
section (see~\ref{example:no_formal_analogue} and~\ref{remark:optimality_sob})
are entailed by examples (see~\ref{example:lebesgue}
and~\ref{example:bunch_of_planes}) based on known scaling properties of
derivatives in Euclidean space.

\begin{example} \label{example:lebesgue}
	Suppose $n$ is an integer, $n \geq 2$, and $n/(n-1) < p \leq \infty$.
	Then,
	\begin{equation*}
		{\textstyle \sup \left \{ ( \mathscr L^n )_{(p)} (f) \with
		\text{$0 \leq f \in \mathscr D ( \mathbf R^n, \mathbf R )$,
		$\spt f \subset \mathbf U(0,1)$, $\int | \mathrm D f | \,
		\mathrm d \mathscr L^n \leq 1$} \right \} = \infty};
	\end{equation*}
	in fact, we fix $0 \leq g \in \mathscr D ( \mathbf R^n, \mathbf R)$
	with $\spt g \subset \mathbf U(0,1)$ and $\int | \mathrm Dg | \,
	\mathrm d \mathscr L^n = 1$, and consider $f_\varepsilon \in \mathscr
	D ( \mathbf R^n, \mathbf R )$ with $f_\varepsilon (x) =
	\varepsilon^{1-n} g( \varepsilon^{-1} x )$ for $x \in \mathbf R^n$ and
	$0 < \varepsilon \leq 1$.
\end{example}

\begin{example} \label{example:bunch_of_planes}
	Suppose $m$ and $n$ are positive integers, $m < n$, and $\Phi$ is the
	set of $(V,f)$ such that $V \in \mathbf {RV}_m ( \mathbf R^n )$, $\| V
	\| \, \mathbf U (0,1) = \boldsymbol \alpha (n)$, $\delta V = 0$, $0
	\leq f \in \mathscr D ( \mathbf R^n, \mathbf R )$, $\spt f \subset
	\mathbf U (0,1)$, and $\int | V \, \mathbf Df | \, \mathrm d \| V \|
	\leq 1$.  Then, we will prove that
	\begin{equation*}
		\sup \left \{ \| V \|_{(p)} (f) \with (V,f) \in \Phi \right \}
		= \infty \quad \text{for $n/(n-1) < p \leq \infty$}.
	\end{equation*}
	We pick $T \in \mathbf G(n,m)$, let $W = \mathscr L^n \times
	\boldsymbol \delta_T \in \mathbf V_m ( \mathbf R^n )$, and recall
	$\delta W = 0$
	from~\ref{example:decomposition}\,\eqref{item:decomposition:no}; in
	particular, the assertion resulting from replacing ``$\mathbf{RV}_m$''
	by ``$\mathbf V_m$'' is a consequence of \ref{example:c1} and
	\ref{example:lebesgue}.  Finally, we approximate $W$ by varifolds~$V
	\in \mathbf{RV}_m ( \mathbf R^n )$ with $\| V \| \, \mathbf U (0,1) =
	\boldsymbol \alpha (n)$, $\delta V = 0$, and $V^{(x)} = \boldsymbol
	\delta_T$ for $x \in \mathbf R^n$.  (Geometrically, each approximating
	varifold $V$ corresponds to a positive multiple of the union of a
	countable collection of affine planes parallel to~$T$.)
\end{example}

\begin{example} [Sobolev inequality vs.~general isoperimetric inequality]
	\label{example:no_formal_analogue}
	\emph{If $m$ and~$n$ are positive integers, $m < n$, $\beta
	= \infty$ if $m=1$, $\beta = m/(m-1)$ if $m>1$, and $0 < d < \infty$, then the supremum
	of the set of all numbers
	\begin{equation*}
		( \| V \| \restrict \{ x \with M(x) \geq d \}
		)_{(\beta)} (f)
	\end{equation*}
	corresponding to $V \in \mathbf{RV}_m ( \mathbf R^n )$ and $f \in
	\mathscr D ( \mathbf R^n, \mathbf R )$
	satisfying $\delta V = 0$, $f \geq 0$, and $\int | V \, \mathbf Df| \,
	\mathrm d \| V \| \leq 1$, where $M$ is associated to $m$, $n$, and
	$V$ as in \ref{miniremark:defMaxFunction}, equals~$\infty$}; in fact,
	assuming $d = \boldsymbol \alpha (m)^{-1} \boldsymbol \alpha (n)$, one may
	take $p = \beta$ in~\ref{example:bunch_of_planes}.
\end{example}
 
\begin{lemma} [Integrating superlevel sets] \label{lemma:federers_fct}
	Suppose $\phi$ measures $X$, $f$ is a nonnegative $\phi$ measurable
	function, $1 \leq p \leq \infty$, and $E (y) = \{ x \with f(x) > y \}$
	for $0 \leq y < \infty$.  Then,
	\begin{equation*}
		{\textstyle \phi_{(p)} (f) \leq \int_0^\infty \phi (
		E(y))^{1/p} \, \mathrm d \mathscr L^1 \, y};
	\end{equation*}
	here $0^{1/p} = 0$ and $\infty^{1/p} = \infty$.
\end{lemma}

\begin{proof}
	Assume $p < \infty$ and $\int_0^\infty \phi (E(y))^{1/p} \, \mathrm d
	\mathscr L^1 \, y < \infty$.  Then, possibly replacing $f(x)$ by $\sup
	\{ 0, f(x)-\varepsilon\}$ for $0 < \varepsilon < \infty$, we may also
	assume $\phi (E(0)) < \infty$.  Abbreviating $f_y = \inf \{ f, y \}$,
	we define $g : \{ y \with 0 \leq y < \infty \} \to \mathbf R$ by
	\begin{equation*}
		g(y) = \phi_{(p)} (f_y) \quad \text{for $0 \leq y < \infty$}.
	\end{equation*}
	Minkowski's inequality (see \cite[2.4.15]{MR41:1976}) yields
	\begin{equation*}
		0 \leq g ( y + \upsilon ) - g ( y ) \leq \phi_{(p)} (
		f_{y+\upsilon} - f_y ) \leq \upsilon \phi ( E(y))^{1/p}
	\end{equation*}
	for $0 \leq y < \infty$ and $0 \leq \upsilon < \infty$.  Therefore,
	$\Lip g < \infty$ and, by \cite[2.9.19]{MR41:1976},
	\begin{equation*}
		0 \leq g'(y) \leq \phi (E(y))^{1/p} \quad \text{for $\mathscr
		L^1$~almost all $0 \leq y < \infty$},
	\end{equation*}
	hence, by \cite[2.4.7,\,9.20]{MR41:1976}, we infer $\phi_{(p)} (f) =
	\lim_{y \to \infty} g(y) = \int_0^\infty g' \, \mathrm d \mathscr
	L^1$.
\end{proof}

\begin{remark} \label{remark:federers_fct}
	The method of the preceding proof is taken from
	\cite[4.5.9\,(18)]{MR41:1976}.
\end{remark}

\begin{theorem} [Sobolev inequality -- with averaging] \label{thm:sob_average}
	Suppose $m$ and $n$ are positive integers, $m \leq n$, $U$ is an open
	subset of $\mathbf R^n$, $V \in \mathbf V_m ( U )$, $\| \delta V \|$
	is a Radon measure, $f \in \mathbf T_{\Bdry U} ( V )$, $E(y) = \{ x
	\with f(x) > y \}$ for $y \in \mathbf R$,  $\| V \| ( E(y) ) < \infty$
	for $0 < y < \infty$,
	\begin{equation*}
		\text{$\beta = \infty$ if $m = 1$}, \quad \text{$\beta =
		m/(m-1)$ if $m>1$},
	\end{equation*}
	$0 < d < \infty$, $r$ is a $\{ t \with 0 < t < \infty \}$~valued $\| V
	\|$~measurable function, $\dmn r \subset U$, $0 < \lambda < 1$, $g :
	\dmn r \to \overline{\mathbf R}$ satisfies
	\begin{equation*}
		g(a) = \sup \big \{ y \with \| V \| ( U \cap \mathbf B
		(a,r(a)) \without E(y) ) \leq \lambda \| V \| ( U \cap \mathbf
		B(a,r(a)) ) \big \}
	\end{equation*}
	for $a \in \dmn r$, and $A = \{ a \with \infty > \| V \| ( U \cap
	\mathbf B (a,r(a))) \geq d \boldsymbol \alpha (m) r(a)^m \}$.

	Then, $g$ is $\| V \|$~measurable and there holds
	\begin{equation*}
		{\textstyle ( \| V \| \restrict A )_{(\beta)} (g) \leq \Gamma
		\big ( \int f \, \mathrm d \| \delta V \| + \int | V \,
		\mathbf Df | \, \mathrm d \| V \| \big )},
	\end{equation*}
	where $\Gamma = (1-\lambda)^{-1} \boldsymbol \beta (n)^{1-1/m}
	\boldsymbol \gamma (m) d^{-1/m}$.
\end{theorem}

\begin{proof}
	Firstly, we use the facts, that the supremum equalling $g(a)$
	remains unchanged when $y$ therein is restricted to be rational and
	that
	\begin{equation*}
		( U \times U ) \cap \{ (a,x) \with |a-x| \leq r(a) \}
	\end{equation*}
	is $\| V \| \times \| V \|$~measurable, to deduce the $\| V
	\|$~measurability of $g$ from Fubini's theorem (see
	\cite[2.6.2]{MR41:1976}).  Next, we define $W_y \in \mathbf V_m (
	\mathbf R^n )$ as in \ref{remark:wy} and let $M_y$ denote the function
	resulting from replacement of $V$ by $W_y$ in the definition of the
	function $M$ in \ref{miniremark:defMaxFunction}.  Whenever $0 < y <
	\infty$, $a \in A$, and $g(a) >y$, we note
	\begin{gather*}
		\| V \| ( U \cap \mathbf B (a,r(a)) ) \leq \| W_y \| \,
		\mathbf B (a,r(a)) + \lambda \| V \| ( U \cap \mathbf B
		(a,r(a))) < \infty, \\
		\|V\| ( U \cap \mathbf{B}(a,r(a)) ) \leq (1-\lambda)^{-1}
		\|W_y\| \, \mathbf{B}(a,r(a)), \\
		\mathbf{B}(a,r(a)) \subset \{x \with M_y (x) \geq (1-\lambda)
		d \}
	\end{gather*}
	by \ref{miniremark:defMaxFunction}.  Therefore, the
	Be\-si\-co\-vich-Federer covering theorem yields
	\begin{align*}
		\| V \| ( A \cap \{ a \with g(a)>y \} ) & = \lim_{i \to
		\infty} \| V \| ( A \cap \{ a \with \text{$g(a)>y$ and
		$r(a)\leq i$} \} ) \\
		& \leq \boldsymbol \beta (n) (1-\lambda)^{-1} \| W_y \| \, \{
		x \with M_y (x) \geq (1-\lambda) d \}
	\end{align*}
	for $0 < y < \infty$, whence we infer, as $\| W_y \| ( \mathbf R^n ) <
	\infty$, that
	\begin{equation*}
		\| V \| ( A \cap \{ a \with g(a)>y \} )^{1/\beta} \leq \Gamma
		\| \delta W_y \| ( \mathbf R^n ) \leq \Gamma \big ( \| \delta
		V \| (E(y)) + \| V \, \partial E(y) \| ( U ) \big  )
	\end{equation*}
	for $\mathscr L^1$ almost all $0 < y < \infty$ by
	\ref{theorem:improvedIsopIneq}, \ref{definition:gamma}, and
	\ref{remark:wy}; here $0^0=0$.  Since $g$ is nonnegative, integrating
	this inequality with respect to $\mathscr L^1$ yields the conclusion
	by means of \ref{lemma:federers_fct}, \cite[2.6.2]{MR41:1976},
	\ref{remark:generalise_sect_8}, and \cite[8.5,\,30]{MR3528825}.
\end{proof}

\begin{remark} \label{remark:optimality_sob}
	If $m < n$, one may not replace $g$ by $f$ in the preceding estimate;
	in fact, in view of \ref{example:c1}, \ref{remark:generalise_sect_9},
	and \cite[9.4]{MR3528825}, one may consider $U = \mathbf
	R^n \cap \mathbf U(0,1)$, $d = 2^{-m} \boldsymbol \alpha (m)^{-1}
	\boldsymbol \alpha (n)$, and $r(a)=2$ for $a \in U$, and take $p =
	\infty$ if $m =1$ and $p = m/(m-1)$ if $m>1$ in
	\ref{example:bunch_of_planes}.
\end{remark}

\begin{theorem} [Sobolev inequality -- rectifiable part] \label{thm:sob_rect}
	Suppose $m$ and $n$ are positive integers, $m \leq n$, $U$ is an open
	subset of $\mathbf R^n$, $V \in \mathbf V_m ( U )$, $\| \delta V \|$
	is a Radon measure, $f \in \mathbf T_{\Bdry U} ( V )$, $\| V \| \, \{
	x \with f (x) > y \} < \infty$ for $0 < y < \infty$,
	\begin{equation*}
		\text{$\beta = \infty$ if $m = 1$}, \quad \text{$\beta =
		m/(m-1)$ if $m>1$},
	\end{equation*}
	$0 < d < \infty$, and $A = \{ a \with \boldsymbol \Theta^m ( \| V \|,
	a ) \geq d \}$.

	Then, there holds
	\begin{equation*}
		{\textstyle ( \| V \| \restrict A )_{(\beta)} (f) \leq \Gamma
		\big ( \int f \, \mathrm d \| \delta V \| + \int | V \,
		\mathbf Df | \, \mathrm d \| V \| \big )},
	\end{equation*}
	where $\Gamma = \boldsymbol \gamma (m) d^{-1/m}$.
\end{theorem}

\begin{proof}
	We define $E(y)$ as in \ref{definition:TGV}.  Moreover, we define $W_y
	\in \mathbf V_m ( \mathbf R^n )$ as in \ref{remark:wy} and let $M_y$
	denote the function resulting from replacement of $V$ by $W_y$ in the
	definition of the function $M$ in \ref{miniremark:defMaxFunction}.
	Since $\boldsymbol \Theta^m ( \| W_y \|, x ) \geq d$ for $\| V
	\|$~almost all~$x \in A \cap E(y)$ by
	\cite[2.8.9,\,18,\,9.11]{MR41:1976}, we conclude
	\begin{equation*}
		\| V \| ( A \cap E (y) ) \leq \| W_y \| \, \{ x \with M_y (x)
		\geq d \} \quad \text{for $0 < y < \infty$}.
	\end{equation*}
	In conjunction with \ref{theorem:improvedIsopIneq},
	\ref{definition:gamma}, and \ref{remark:wy}, we infer, as $\| W_y \| (
	\mathbf R^n ) < \infty$, that
	\begin{equation*}
		\| V \| ( A \cap E(y) )^{1/\beta} \leq \Gamma \| \delta W_y \|
		( \mathbf R^n ) \leq \Gamma \big ( \| \delta V \| (E(y)) + \|
		V \, \partial E(y) \| ( U ) \big  )
	\end{equation*}
	for $\mathscr L^1$ almost all $0 < y < \infty$; here $0^0=0$.
	Integrating this inequality yields the conclusion by means of
	\ref{lemma:federers_fct}, \cite[2.6.2]{MR41:1976},
	\ref{remark:generalise_sect_8}, and \cite[8.5,\,30]{MR3528825}.
\end{proof}

\begin{theorem} [Poincar{\'e} inequality in a ball -- zero boundary values] \label{thm:poincareTypeIneq}
	Suppose $m$ and $n$ are positive integers, $m \leq n$, $a \in \mathbf
	R^n$, $0 < r < \infty$, $V \in \mathbf V_m ( \mathbf U (a,r) )$, $\|
	\delta V \|$ is a Radon measure, and $f \in \mathbf T_{\Bdry \mathbf U
	(a,r)} ( V )$.

	Then, there holds
	\begin{equation*}
		{\textstyle \boldsymbol{\alpha} (m)^{-1/m} r^{-1} \int f
		\,\mathrm d \|V\| \leq \boldsymbol{\gamma}(m) \left(
		\textstyle\int f\,\mathrm d \| \delta V \| + \int
		| V \,\mathbf{D}f| \, \mathrm d \| V \| \right)}.
	\end{equation*}
\end{theorem}

\begin{proof}
	Define $E(y) = \{ x \with f(x) > y \}$ for $0 < y < \infty$.  In view
	of \ref{remark:wy}, we apply~\ref{theorem:poincareTypeIneq} with
	$V$ replaced by $W_y$ to obtain
	\begin{equation*}
		\boldsymbol \alpha (m)^{-1/m} r^{-1} \| V \| ( E(y)) \leq
		\boldsymbol \gamma (m) \big ( \| \delta V \| ( E(y)) + \| V \,
		\partial E(y) \| \, \mathbf U (a,r) \big )
	\end{equation*}
	for $\mathscr L^1$ almost all $0 < y < \infty$.  Integrating this
	inequality with respect to $\mathscr L^1$ yields the conclusion by
	means of Fubini's theorem (see \cite[2.6.2]{MR41:1976}) and the coarea
	formula (see \ref{remark:generalise_sect_8} and
	\cite[8.5,\,30]{MR3528825}).
\end{proof}

\begin{remark}
	In view of \ref{example:c1}, \ref{remark:generalise_sect_9}, and
	\cite[9.4]{MR3528825}, there is no similar control of $\| V \|_{(p)}
	(f)$ involving a number depending only on $m$ and $p$, for any $p>1$
	by~\ref{example:bunch_of_planes}.
\end{remark}

\addcontentsline{toc}{section}{\numberline{}References}

\medskip \noindent \textsc{Affiliations}

\medskip \noindent Ulrich Menne \smallskip \newline Institute of Mathematics,
University of Leipzig \newline Augustusplatz 10, 04109 \textsc{Leipzig},
\textsc{Germany} \smallskip \newline Max Planck Institute for Mathematics in
the Sciences \newline Inselstra{\ss}e 22,  04103 \textsc{Leipzig},
\textsc{Germany}

\medskip \noindent Christian Scharrer \smallskip \newline
Mathematics Institute, Zeeman Building, University of Warwick \newline
\textsc{Coventry} CV4 7AL, \textsc{Great Britain}

\medskip \noindent \textsc{Email addresses}

\medskip \noindent
\href{mailto:Ulrich.Menne@math.uni-leipzig.de}{Ulrich.Menne@math.uni-leipzig.de}
\quad
\href{mailto:C.Scharrer@warwick.ac.uk}{C.Scharrer@warwick.ac.uk}

\end{document}